\RequirePackage{fix-cm}
\documentclass[smallextended]{svjour3}       
\smartqed  
\usepackage{graphicx}
\usepackage{amsmath}
\usepackage{amssymb}
\usepackage{enumerate}
\usepackage{color}

\usepackage[colorlinks=true, allcolors=blue]{hyperref}

\begin{document}

\title{The geometry of Gauss map and shape operator in simply isotropic and pseudo-isotropic spaces\thanks{This is a pre-print of an article published in Journal of Geometry (2019) \textbf{110}: 31. The final authenticated version is available online at: \url{https://doi.org/10.1007/s00022-019-0488-9}}}

\titlerunning{The geometry of Gauss map and shape operator in isotropic spaces}        

\author{Luiz C. B. da Silva}


\institute{L. C. B. da Silva \at Department of Physics of Complex Systems, Weizmann Institute of Science,\\
              Rehovot 7610001, Israel \\
              \email{luiz.da-silva@weizmann.ac.il}
}
\date{Received: 2018 / Accepted: 2019}

\maketitle

\begin{abstract}
In this work, we are interested in the differential geometry of surfaces in simply isotropic $\mathbb{I}^3$ and pseudo-isotropic $\mathbb{I}_{\mathrm{p}}^3$ spaces, which consists of the study of $\mathbb{R}^3$ equipped with a degenerate metric such as $\mathrm{d}s^2=\mathrm{d}x^2\pm\mathrm{d}y^2$. The investigation is based on previous results in the simply isotropic space [B. Pavkovi\'c, Glas. Mat. Ser. III \textbf{15}, 149 (1980); Rad JAZU \textbf{450}, 129 (1990)], which point to the possibility of introducing an isotropic Gauss map taking values on a unit sphere of parabolic type and of defining a shape operator from it, whose determinant and trace give the known relative Gaussian and mean curvatures, respectively. Based on the isotropic Gauss map, a new notion of connection is also introduced, the \emph{relative connection} (\emph{r-connection}, for short). We show that the new curvature tensor in both $\mathbb{I}^3$ and $\mathbb{I}_{\mathrm{p}}^3$ does not vanish identically and is directly related to the relative Gaussian curvature. We also compute the Gauss and Codazzi-Mainardi equations for the $r$-connection and show that $r$-geodesics on spheres of parabolic type are obtained via intersections with planes passing through their center (focus). Finally, we show that admissible pseudo-isotropic surfaces are timelike and that their shape operator may fail to be diagonalizable, in analogy to Lorentzian geometry. We also prove that the only totally umbilical surfaces in $\mathbb{I}_{\mathrm{p}}^3$ are planes and spheres of parabolic type and that, in contrast to the $r$-connection, the curvature tensor associated with the isotropic Levi-Civita connection vanishes identically for \emph{any} pseudo-isotropic surface, as also happens in simply isotropic space.

\end{abstract}

\section{Introduction}

The three dimensional ($3d$) simply isotropic $\mathbb{I}^3$ and pseudo-isotropic $\mathbb{I}_{\mathrm{p}}^3$ spaces are examples of $3d$ Cayley-Klein (CK) geometries \cite{daSilvaArXivIso2017,Giering1982,Sachs1987}. Basically, a CK geometry is the study of the properties in projective space $\mathbb{P}^3$ invariant by the action of the subgroup of projectivities that preserves the so-called \emph{absolute figure}. In our case of interest, the absolute figure is given by a plane at infinity and a degenerate quadric of index zero \cite{Sachs1990,StrubeckerSOA1941} or one \cite{daSilvaArXivIso2017}. 
From the differential viewpoint, we are essentially led to the study of $\mathbb{R}^3$ equipped with a degenerate metric of index 0 or 1: $\mathrm{d}s^2=\mathrm{d}x^2\pm\mathrm{d}y^2$. Besides its mathematical interest, see e.g. \cite{AydinGMJ2017,daSilvaArXivIso2017,Sachs1990}, isotropic geometry also finds applications in economics \cite{AydinTJM2016,chenKJM2014}, elasticity \cite{pottmann2009laguerre}, and in image processing \cite{koenderink2002image,PottmannCAGD1994}. The geometry of curves and surfaces in $\mathbb{I}^3$ was {began} by Strubecker in a series of papers \cite{StrubeckerSOA1941,StrubeckerMZ1942II,StrubeckerMZ1942III}, see also \cite{Sachs1990}, while the respective theory in $\mathbb{I}^3_{\mathrm{p}}$ has been recently initiated in \cite{AydinArXivPseuIso,daSilvaArXivIso2017}. 

It is known that the Riemann curvature tensor induced from the isotropic Levi-Civita connection vanishes for any surface in $\mathbb{I}^3$ \cite{Sachs1990}.  Despite that, the concept of a second fundamental form is not  trivial and allows the introduction of an alternative notion of Gaussian curvature $K$, named \emph{relative Gaussian curvature},  whose expression in local coordinates and interpretation via normal curvatures are analogous to the {Euclidean versions}. Amazingly, $K$ is expressed as the ratio between the area of a region under the (parabolic) spherical image and the isotropic area on the surface \cite{Sachs1990}. Finally, in the 1980s \cite{PavkovicGM1980} Pavkovi\'c interpreted $K$ in terms of a shape operator defined with respect to a unit sphere of parabolic type. These results suggest that the relative Gaussian curvature is a proper substitute for the intrinsic one. However, to the best of our knowledge, a possible relation with a curvature tensor has not been investigated yet.

In this work we push the results in \cite{PavkovicGM1980,PavkovicRadJAZU1990} further and extend them to $\mathbb{I}_{\mathrm{p}}^3$. After preliminaries in Sect. 2, a Gauss map in both $\mathbb{I}^3$ and $\mathbb{I}_{\mathrm{p}}^3$ taking values on a unit sphere of parabolic type, but slightly distinct from that of \cite{PavkovicGM1980}, is introduced in Sect. 3, from which a shape operator and the relative Gaussian and mean curvatures are defined. Following \cite{PavkovicRadJAZU1990}, in Sects. 4 and 5,  a new connection, the \emph{relative connection} (\emph{r-connection}), is introduced and we show that the respective curvature tensor does not vanish identically and is related to the relative Gaussian curvature. We compute the respective Gauss and Codazzi-Mainardi equations as well. Relative geodesics are defined as autoparallel curves of the $r$-connection and, as an example, we show that $r$-geodesics on spheres of parabolic type are obtained by intersections with planes passing through their center (focus). In Sect. 5 we turn to surfaces in $\mathbb{I}_{\mathrm{p}}^3$, where we first establish some elementary results. We prove that every (admissible) surface is timelike, show that their shape operator may fail to be diagonalizable, and prove that totally umbilical surfaces in $\mathbb{I}_{\mathrm{p}}^3$ are planes and spheres of parabolic type only. Finally, we show that as also happens in $\mathbb{I}^3$, the curvature tensor associated with the isotropic Levi-Civita connection vanishes identically for \emph{any} surface in $\mathbb{I}_{\mathrm{p}}^3$, in contrast to what is verified for the $r$-connection.

In the remaining of this work we shall use the Einstein  convention of summing on repeated indexes, e.g., $\Gamma_{ij}^k\mathbf{x}_{k}:=\sum_{k=1}^2\Gamma_{ij}^{k}\mathbf{x}_{k}$.

\section{Preliminaries}

In the spirit of Klein's Erlangen Program, the simply isotropic $\mathbb{I}^3$ and pseudo-isotropic $\mathbb{I}_{\mathrm{p}}^3$ geometries are the study of those properties in $\mathbb{R}^3$ invariant by the action of the 6-parameter groups $\mathcal{B}_6$ \cite{Sachs1990,StrubeckerSOA1941} and $\mathcal{B}_6^{\mathrm{p}}$ \cite{daSilvaArXivIso2017} given, respectively, by
\begin{equation}
\left\{
\begin{array}{ccc}
\bar{x} & = & a + x\,\cos\phi-y\,\sin\phi\\ 
\bar{y} & = & b + x\,\sin\phi+y\,\cos\phi\\
\bar{z} & = & c + c_1x+c_2y+z\\
\end{array}
\right. \mbox{ and }\left\{
\begin{array}{ccc}
\bar{x} & = & a + x\,\cosh\phi+y\,\sinh\phi\\ 
\bar{y} & = & b + x\,\sinh\phi+y\,\cosh\phi\\
\bar{z} & = & c + c_1x+c_2y+z\\
\end{array}
\right.\,,\label{eq::SemiIsoGroupSB6}
\end{equation}
where, $a,b,c,c_1,c_2,\phi\in\mathbb{R}$. In other words, $\mathcal{B}_6$ and $\mathcal{B}_6^{\mathrm{p}}$ give our rigid motions.

Observe that on the $xy$-plane these geometries look exactly like the Euclidean $\mathbb{E}^2$ and Lorentzian $\mathbb{E}_1^2$ plane geometries. The projection of a vector $\mathbf{u}=(u^1,u^2,u^3)\in \mathbb{I}^3$, or $\mathbb{I}_{\mathrm{p}}^3$, on the $xy$-plane is called the \emph{top view} of $\mathbf{u}$ and we shall denote it by $\widetilde{\mathbf{u}}=(u^1,u^2,0)$. The top view concept plays a fundamental role in the simply and pseudo isotropic spaces, since the $z$-direction is preserved under the action of $\mathcal{B}_6$ and $\mathcal{B}_6^{\mathrm{p}}$. A line with this direction is called an \emph{isotropic line} and a plane containing an isotropic line is an \emph{isotropic plane}.

In addition, one may introduce a \emph{simply isotropic} and a \emph{pseudo-isotropic inner product} between two vectors $\mathbf{u}=(u^1,u^2,u^3)$ and $\mathbf{v}=(v^1,v^2,v^3)$ as
\begin{equation}
\langle\mathbf{u},\mathbf{v}\rangle_{z} = u^1v^1+u^2v^2
\mbox{ and }
\langle\mathbf{u},\mathbf{v}\rangle_{pz} = u^1v^1-u^2v^2,
\end{equation}
respectively\footnote{The index $z$ is here to emphasize that $z$ is the isotropic (degenerate) direction.}. These inner products induce a (semi) norm in a natural way: 
\begin{equation}
\Vert \mathbf{u}\Vert_z = \sqrt{\langle\mathbf{u},\mathbf{u}\rangle_z} = \Vert\widetilde{\mathbf{u}}\Vert
\mbox{ and }
\Vert \mathbf{u}\Vert_{pz} = \sqrt{\vert\langle\mathbf{u},\mathbf{u}\rangle_{pz}\vert} = \Vert\widetilde{\mathbf{u}}\Vert_1,
\end{equation}
respectively. Here, $\Vert\cdot\Vert$ and $\Vert\cdot\Vert_1$ are the Euclidean and Lorentzian norms {induced by
$
\langle\mathbf{u},\mathbf{v}\rangle=u^1v^1+u^2v^2+u^3v^3$ and $\langle\mathbf{u},\mathbf{v}\rangle_1=u^1v^1-u^2v^2+u^3v^3$, respectively, and whose corresponding vector products are $\times$ and $\times_1$: notice that
$\mathbf{u}\times_1\mathbf{v}=(u^2v^3-u^3v^2,u^1v^3-u^3v^1,u^1v^2-u^2v^1)$.}

Finally, since the isotropic metrics are degenerate, the distance from a point $\mathbf{u}=(u^1,u^2,u^3)$ to $\mathbf{v}=(u^1,u^2,v^3)$ is zero {since $\tilde{\mathbf{u}}=\tilde{\mathbf{v}}$. In such cases, one may define a co-metric by using the codistance $\mathrm{cd}(\mathbf{u},\mathbf{v})=\vert v^3-u^3\vert.$}

\begin{remark}
The codistance $\mathrm{cd}(\cdot,\cdot)$ is a secondary concept and it is invariant by $\mathcal{B}_6$, or $\mathcal{B}_6^{\mathrm{p}}$, only when applied to isotropic vectors \cite{daSilvaArXivIso2017,Sachs1990}. One should not see it as part of the definition of a (pseudo) isotropic distance. Indeed, the function { $G$ defined as $G(\mathbf{u},\mathbf{v})=
\langle\mathbf{u},\mathbf{v}\rangle_z$ if  $\tilde{\mathbf{u}}\not=0$ or $\tilde{\mathbf{v}}\not=0$, and as
$G(\mathbf{u},\mathbf{v})=u^3v^3$ if $\tilde{\mathbf{u}}=\tilde{\mathbf{v}}=0$,
 is not bilinear and thus \emph{can not} be a metric: e.g.,} $G((1,0,1)+(0,0,1),(0,0,1))=0$, but $G((1,0,1),(0,0,1))+G(0,0,1),(0,0,1))=1\not=0$. 
\end{remark}

\begin{remark}
{If instead of $x_0=x_1=0$ and $x_0=x_2=0$, we choose the pair of lines $x_0=x_1\pm\, x_2=0$  for the pseudo-isotropic absolute figure}, then we obtain a different group of pseudo-isotropic rigid motions \cite{daSilvaArXivIso2017}, which coincides with the choice made in the classical literature (see the following remark). These groups however lead to the same geometry and can be related by a convenient coordinate change on the top view plane \cite{daSilvaArXivIso2017}. 
\end{remark}

\begin{remark}[Notation and terminology]
In \cite{StrubeckerCrelle1938,StrubeckerJDMV1938} Strubecker used the metric $\mathrm{d}s^2=\mathrm{d}x\,\mathrm{d}y$ and denoted the corresponding geometry by $I_3$. Notice this metric is equivalent to $\mathrm{d}s^2=\mathrm{d}x^2-\mathrm{d}y^2$ \cite{daSilvaArXivIso2017}, Subsect. 6.1. Here, the lines in the absolute figure are $f_1:\,x_0=x_1=0$ and $f_2:\,x_0=x_2=0$. In \cite{StrubeckerSOA1941}, however, Strubecker started to consider the two intersecting lines as being $f_{1,2}:x_0=x_1\pm\mathrm{i}x_2=0$, which leads to the distinct metric $\mathrm{d}s^2=\mathrm{d}x^2+\mathrm{d}y^2$, but continues to call the space isotropic and keeps denoting it by $I_3$. Around the 1930s, Lense seems to be the first to pay more attention to the ``degree of isotropy'' \cite{LenseMA1939}: e.g., when introducing the doubly isotropic space, Brauner based his terminology on Lense's work \cite{BraunerCrelle1966}. In \cite{SachsGD1989} Sachs denoted the pseudo-isotropic space, which comes from Strubecker's first choice, by $\tilde{I}_3^{(1)}$ and this same notation were employed by Husty and R{\"o}schel five years earlier \cite{HustyCMSJB1984}. In his book, however, Sachs denoted the pseudo-isotropic space by $I_3^{(1)P}$ \cite{Sachs1990} and this notation were also used recently by M\'esz\'aros \cite{MeszarosMP2007}.  In our work, for the ease of notation, we write $\mathbb{I}^3$ for the simply isotropic space and $\mathbb{I}_{\mathrm{p}}^3$ for the pseudo-isotropic space, since modern geometry texts usually put the dimension as an upper index.
\end{remark}

\subsection{Spheres in {isotropic spaces}}

In the following sections we shall use spheres of parabolic type in order to define a Gauss map and a shape operator for surfaces in isotropic spaces. \emph{Simply isotropic} and \emph{pseudo-isotropic spheres} {are defined} as connected and irreducible surfaces of degree 2 given by the 4-parameter families 
\begin{equation}
{(x^2\pm y^2)+2c_1x+2c_2y+2c_3z+c_4=0,\,c_i\in\mathbb{R},}\label{eq::IsoSpheres}
\end{equation}
{where the sign is $+$ in $\mathbb{I}^3$ and $-$ in $\mathbb{I}_{\mathrm{p}}^3$ \cite{daSilvaArXivIso2017,Sachs1990}.}  In addition, up to a rigid motion, we can express an isotropic sphere in one of the two normal forms below:
\newline
1. \emph{Spheres of parabolic type}: 
\begin{equation}
{z = \frac{1}{2p}(x^2\pm y^2)-\frac{p}{2}\,\mbox{ with }\,p\not=0,} \label{eq::IsoParSphere}
\end{equation}
{where the sign is $+$ in $\mathbb{I}^3$ and $-$ in $\mathbb{I}_{\mathrm{p}}^3$; and}
\newline
2. \emph{Spheres of cylindrical type}: 
\begin{equation}
{\left\{
\begin{array}{lc}
x^2+y^2=r^2 & \mbox{ in }\mathbb{I}^3\\
x^2-y^2=\pm\, r^2 & \mbox{ in }\mathbb{I}_{\mathrm{p}}^3\\
\end{array}
\right.,\mbox{ with }\,r>0.}
\end{equation}

\begin{remark}
The quantities $p$ and $r$ are isotropic invariants. Moreover, spheres of cylindrical type are precisely the set of points equidistant from a given center\footnote{Observe that the center $P$ is not uniquely defined since any other point $Q$ with the same top view as $P$, i.e., $\tilde{Q}=\tilde{P}$, is also a center.}. However, they do not constitute ``good'' surfaces in isotropic geometry, since their tangent planes are isotropic, i.e., in the terminology of the following section, they are not admissible surfaces. On the other hand, spheres of parabolic type {are understood to be centered at their focus. They} are admissible surfaces, have constant Gaussian and mean curvatures, $K=\frac{1}{p^2}$ and $H=\frac{1}{p}$, and {are the only totally umbilical surfaces in addition to planes} (Props. \ref{prop::GaussMeanCurvParabSph} and \ref{prop::SphrParAreTU}). 
\end{remark}

\section{Surfaces in isotropic spaces}

Now we discuss the differential geometry of surfaces in isotropic spaces. Further details concerning the geometry in $\mathbb{I}^3$ can be found in \cite{Sachs1990}. 
On the other hand, the study of surfaces in pseudo-isotropic geometry was initiated in Ref. \cite{AydinArXivPseuIso}.
\begin{definition}
Let $\mathbb{M}^3$ be $\mathbb{I}^3$ or $\mathbb{I}_{\mathrm{p}}^3$. Then, $\mathbf{x}:S\to\mathbb{M}^3$ is an \emph{admissible} surface if its tangent planes $T_qS$ are all non-isotropic.
\end{definition}

Let $g$ be the metric in $S\subset\mathbb{M}^3$ induced by the immersion $\mathbf{x}$, i.e., $g(u,v)=\langle\mathrm{d}\mathbf{x}(u),\mathrm{d}\mathbf{x}(v)\rangle_z$, or $g(u,v)=\langle\mathrm{d}\mathbf{x}(u),\mathrm{d}\mathbf{x}(v)\rangle_{pz}$. In a local parameterization $(u^1,u^2)\in\mathcal{U}\subseteq\mathbb{R}^2\mapsto\mathbf{x}(u^1,u^2)\in S$, the coefficients of the first fundamental form $\mathrm{I}$ are $g_{ij}=\langle\mathbf{x}_i,\mathbf{x}_j\rangle_z$, or $g_{ij}=\langle\mathbf{x}_i,\mathbf{x}_j\rangle_{pz}$, where $\mathbf{x}_i=\partial \mathbf{x}/\partial u^i$. We shall see in Prop. \ref{prop::PseudoIsoSurfAreTimelike} that for any admissible surface in $\mathbb{I}^3_{\mathrm{p}}$ the induced metric is non-degenerate with index 1, i.e., every surface is timelike. In addition, any admissible surface $\mathbf{x}(u^1,u^2)=(x^1(u^1,u^2),x^2(u^1,u^2),x^3(u^1,u^2))$ satisfies 
\begin{equation}
\frac{\partial(x^1,x^2)}{\partial(u^1,u^2)}=(x_1^1x_2^2-x_2^1x_1^2)\not=0,
\end{equation}
and, {then, can be parameterized as a graph, known as the \emph{normal form}:} 
\begin{equation}\label{eq::1stFundFormNormalPar}
{\mathbf{x}(u^1,u^2)=(u^1,u^2,f(u^1,u^2))\Rightarrow
\left\{
\begin{array}{cc}
\mathrm{I} = (\mathrm{d}u^1)^2+(\mathrm{d}u^2)^2 &\mbox{ in }\mathbb{I}^3\\
\mathrm{I}=(\mathrm{d}u^1)^2-(\mathrm{d}u^2)^2 &\mbox{ in }\mathbb{I}^3_{\mathrm{p}}\\
\end{array}
\right..}
\end{equation}

\subsection{Isotropic Gauss map and shape operator}

Denoting $\mathbf{x}=(x^1,x^2,x^3)$ and $\mathbf{x}_i=(x_i^1,x_i^2,x_i^3)$, we   introduce the notations
\begin{equation}
X=\left(
\begin{array}{ccc}
x_1^1 & x_1^2 & x_1^3\\[3pt]
x_2^1 & x_2^2 & x_2^3\\
\end{array}
\right),\, X_{ij}=\det\left(
\begin{array}{cc}
x_1^i & x_1^j \\[3pt]
x_2^i & x_2^j \\
\end{array}
\right).\label{def::DefXandXij}
\end{equation}
It follows that $\det(g_{ij}){=\Vert\tilde{\mathbf{x}}_1\times\tilde{\mathbf{x}}_2\Vert^2}=(X_{12})^2>0$ for $\mathbb{M}^3=\mathbb{I}^3$ and $\det(g_{ij})={=-\Vert\tilde{\mathbf{x}}_1\times_1\tilde{\mathbf{x}}_2\Vert_1^2}=-(X_{12})^2<0$ for $\mathbb{M}^3=\mathbb{I}_{\mathrm{p}}^3$. Notice that $X_{12}(q)=0$ would mean an isotropic tangent plane at $q$.

\begin{proposition}\label{prop::PseudoIsoSurfAreTimelike}
In  $\mathbb{I}_{\mathrm{p}}^3$ every admissible surface is timelike. In addition, there exists no spacelike surface and the only lightlike ones are  non-admissible.
\end{proposition}
\begin{proof}
In $\mathbb{I}_{\mathrm{p}}^3$ all admissible surfaces satisfy $\det { g_{ij} }=-(X_{12})^2<0$, which shows that they should be timelike, i.e., $g_{ij}$ is non-degenerate and has index 1. In particular, there is no spacelike surface. Finally, a non-admissible surface gives $\det g_{ij}=0$. Then,  a surface is lightlike if, and only if, it is non-admissible.
\qed
\end{proof}

Let $\Sigma^2$ be the unit sphere of parabolic type in $\mathbb{M}^3$ given by
\begin{equation}
\Sigma^2 = \{(x,y,z)\in\mathbb{M}^3:z = -\frac{1}{2}(x^2\pm y^2)+\frac{1}{2}\},\label{eq::UnitSimpIsoSphPar}
\end{equation}
where the sign is $+$ in $\mathbb{I}^3$ and $-$ in $\mathbb{I}_{\mathrm{p}}^3$.
The sphere $\Sigma^2$ will play a role in isotropic geometry analogous to that of $\mathbb{S}^2$ in $\mathbb{E}^3$ and of $\mathbb{S}_1^2$ in  $\mathbb{E}_1^3$. However, there is no isotropic counterpart of $\mathbb{H}_0^2$, since any surface in $\mathbb{I}_{\mathrm{p}}^3$ is timelike, Prop. \ref{prop::PseudoIsoSurfAreTimelike}.

\begin{definition}
Denoting by $\{\mathbf{e}_i\}_{i=1}^3$ the canonical basis of $\mathbb{R}^3$, where $\mathbf{e}_3$ is isotropic, { the \emph{isotropic Gauss map} $\xi:S\to\Sigma^2$ is defined in $\mathbb{M}^3=\mathbb{I}^3$ as}
\begin{equation}\label{def::IsoGaussMap}
\xi(u^1,u^2) = \frac{X_{23}}{X_{12}}\mathbf{e}_1+\frac{X_{31}}{X_{12}}\mathbf{e}_2+\frac{1}{2}\left\{1-\left[\left(\frac{X_{23}}{X_{12}}\right)^2+\left(\frac{X_{31}}{X_{12}}\right)^2\right]\right\}\mathbf{e}_3,
\end{equation}
 and in $\mathbb{M}^3=\mathbb{I}_{\mathrm{p}}^3$ as  
\begin{equation}\label{def::PseudoIsoGaussMap}
\xi^{\mathrm{p}}(u^1,u^2) = \frac{X_{23}}{X_{12}}\mathbf{e}_1+\frac{X_{13}}{X_{12}}\mathbf{e}_2+\frac{1}{2}\left\{1-\left[\left(\frac{X_{23}}{X_{12}}\right)^2-\left(\frac{X_{13}}{X_{12}}\right)^2\right]\right\}\mathbf{e}_3.
\end{equation}
\end{definition}

{This definition is inspired by Pavkovi\'c's findings in $\mathbb{I}^3$ \cite{PavkovicGM1980}. Here, however, we made a translation in $z$. This is crucial to guarantee that $\{\mathbf{x}_1,\mathbf{x}_2,\xi\}$ is linearly independent for \emph{any} admissible surface, see Eq. (\ref{eq::LIconditionForGaussMap}). The consequence is that the relative connection, to be introduced in Subsects. \ref{subsect::RelDGsimplyIsoSpc} and \ref{subsect::RelDGpseudoIsoSpc}, can be always properly defined. In addition, such a modification also implies that
 $\Vert\tilde{\xi}\Vert^2+\xi^3>0$, which avoids the appearance of singularities without geometric significance in, e.g., Eqs. (\ref{eq::RelationRHOijHij}) and \eqref{eq::IsoRdabc}. Indeed, if they were ``geometric", we would expect to see an effect manifest in the curvatures, but from Defs. \ref{def::ShapeOp} and \ref{def::IsoGaussAndMeanCurv}, the curvatures come from a derivative, which is invariant by translations, while the singularities disappear for a proper choice of constant. Translations in the definition of $\xi$ may be interpreted in terms of the problem of where to center isotropic spheres. Our choice corresponds to centering spheres at their focus: e.g., $\Sigma^2$ above is centered at the origin of the coordinate system. (Nicely, spherical geodesics come from intersections with planes passing through the center of the sphere, in analogy to $\mathbb{E}^3$ and $\mathbb{E}_1^3$, Examples \ref{Exe::r-GeodOnASphere} and \ref{Exe::r-GeodOnAPseudoIsoSphere}.)}  
\begin{remark}
{In  \cite{PavkovicRadJAZU1990}, Pavkovi\'c addresses the issue of introducing connections in $\mathbb{I}^3$ distinct from the usual one coming from the isotropic normal $\mathcal{N}=(0,0,1)$. There, it is discussed the relations between the second fundamental form and connection coefficients and also the corresponding geodesics. Notice that Pavkovi\'c also points to the need of a \emph{linearly independence condition} for the vector field $\mathbf{V}$ defining a connection, $\varphi\not=0$ in his Eq. (4). However, this is not met by the Gauss map $N_r$ in \cite{PavkovicGM1980}: $N_r=0$ whenever $T_qS$ is the $xy$-plane. Here,  this problem is corrected, as discussed above, and we go a step further in computing the corresponding curvature tensor and Gauss-Codazzi-Mainardi equations in both $\mathbb{I}^3$ and $\mathbb{I}_{\mathrm{p}}^3$. As an example, we describe spherical geodesics.}
\end{remark}

Note that the top view of $\xi$ coincides with that of
\begin{equation}
N_h=\frac{\mathbf{x}_1\times\mathbf{x}_2}{\Vert\widetilde{\mathbf{x}}_1\times\widetilde{\mathbf{x}}_2\Vert}=\frac{X_{23}}{X_{12}}\mathbf{e}_1+\frac{X_{31}}{X_{12}}\mathbf{e}_2+\mathbf{e}_3.\label{eq::DefIsoNh}
\end{equation}
Observe the similarity between $N_h$ and the Euclidean Gauss map $\xi_{eucl}=\mathbf{x}_1\times\mathbf{x}_2\,\Vert\mathbf{x}_1\times\mathbf{x}_2\Vert^{-1}$, but in comparison with $\xi$, the $z$-coordinate of $N_h$ was ``adjusted" to guarantee  $\xi\circ\mathbf{x}\in\Sigma^2$. (Here, we may assume that $X_{12}>0$ by exchanging $u^1\leftrightarrow u^2$ if necessary.) The same is true in $\mathbb{I}_{\mathrm{p}}^3$ as well:
\begin{equation}
N_h=\frac{\mathbf{x}_1\times_1\mathbf{x}_2}{\Vert\widetilde{\mathbf{x}}_1\times_1\widetilde{\mathbf{x}}_2\Vert_1}=\frac{X_{23}}{X_{12}}\mathbf{e}_1+\frac{X_{13}}{X_{12}}\mathbf{e}_2+\mathbf{e}_3\Rightarrow \widetilde{N}_h=\widetilde{\xi}^{\mathrm{p}}.\label{eq::DefPseudoNh}
\end{equation}

\begin{definition}\label{def::ShapeOp}
The {\emph{isotropic shape operator} $L_q$ (or \emph{Weingarten map})} is defined as
\begin{equation}
L_q(w_q)=\left\{
\begin{array}{ccc}
-D_{w_q}\,\xi &,& \forall\,w_q\in T_qS\subset T_q\mathbb{I}^3\\[3pt]
-D_{w_q}\,\xi^{\mathrm{p}} &,& \forall\,w_q\in T_qS\subset T_q\mathbb{I}_{\mathrm{p}}^3\\
\end{array}
\right.,
\end{equation}
where $D$ denotes the usual directional derivative in $\mathbb{R}^3$, i.e., $D_{w_q}\xi=(\xi\circ\gamma)'(0)$ if $\gamma$ is a curve with $\gamma(0)=q$ and $\gamma'(0)=w_q$. 
\end{definition}

Following similar steps to those of Ref. \cite{PavkovicGM1980}, it can be shown that 
\begin{equation}
L_q(\mathbf{x}_i)=-\frac{1}{X_{12}}\det\left(
\begin{array}{lr}
A_{i} & x_2^1\\[3pt]
B_{i} & x_2^2\\
\end{array}
\right)\mathbf{x}_1-\frac{1}{X_{12}}\det\left(
\begin{array}{lr}
x_1^1 & A_{i}\\[3pt]
x_1^2 & B_{i}\\
\end{array}
\right)\mathbf{x}_2,
\end{equation}
where $A=X_{23}/X_{12}$ and $B=X_{31}/X_{12}$ for $\mathbb{I}^3$ or $B=X_{13}/X_{12}$ for $\mathbb{I}_{\mathrm{p}}^3$. The equation above means that $L_q$ can be seen as a linear operator on $T_qS$. The planes $T_qS$ and $T_{\xi(q)}\Sigma^2$ are then parallel and can be canonically identified.

{Since our Gauss map differs from that of  \cite{PavkovicGM1980} only by a constant translation, see discussion and remark following Def. \ref{def::IsoGaussMap}, quantities such as the second fundamental form, normal curvature, and Gaussian and mean curvatures are the same. The \emph{isotropic second fundamental form} $\mathrm{II}$ is}
\begin{equation}
\forall\,u_q,v_q\in T_qS,\,\mathrm{II}(u_q,v_q) = \mathrm{I}(L_q(u_q),v_q).
\end{equation}
{whose coefficients  in local coordinates $\mathbf{x}:S\to\mathbb{M}^3$ can be written as $h_{ij}=\mathrm{II}(\mathbf{x}_i,\mathbf{x}_j)$. In terms of $N_h$, the coefficients of $\mathrm{II}$ are also written as}
\begin{equation}
h_{ij} = \frac{\det(\mathbf{x}_1,\mathbf{x}_2,\mathbf{x}_{ij})}{\det(\widetilde{\mathbf{x}}_1,\widetilde{\mathbf{x}}_2)}=\langle N_h,\mathbf{x}_{ij}\rangle.
\end{equation}
Finally, the \emph{normal curvature} $\kappa_n$ in the direction of a unit vector $w_q\in T_qS$ is 
\begin{equation}
\kappa_n(q,w_q)=\mathrm{II}(w_q,w_q)=\mathrm{I}(L_q(w_q),w_q).
\end{equation}
If $\gamma\subset S$ is a curve with $\gamma'(0)=w_q$, then the equation above is just the component of the acceleration $\gamma''$ in the direction of $N_h$, which is precisely the isotropic normal curvature \cite{Sachs1990}, p. 155. The same reasoning applies to $\mathbb{I}^3_{\mathrm{p}}$ \cite{AydinArXivPseuIso}.

\begin{definition}\label{def::IsoGaussAndMeanCurv}
The \emph{isotropic Gaussian} and \emph{mean curvatures} of an admissible surface $S\subset\mathbb{M}^3$ are respectively defined as
\begin{equation}
K(q) = \det(L_q)\mbox{ and }H(q)=\frac{1}{2}\mathrm{tr}(L_q).
\end{equation}
\end{definition}

If we write $L_q(\mathbf{x}_i)=-A_i^k\,\mathbf{x}_k$ in local coordinates, then
\begin{equation}
h_{ij} = \mathrm{I}(L_q(\mathbf{x}_i),\mathbf{x}_j)=-A_i^k\,\mathrm{I}(\mathbf{x}_k,\mathbf{x}_j) = -A_i^k\,g_{kj}.\label{eq::CoefShapeOpe}
\end{equation}
From this relation, it follows that $-A_i^k=g^{kj}h_{ji}$ and, therefore, we can write
\begin{equation}
K = \frac{h_{11}h_{22}-h_{12}^2}{g_{11}g_{22}-g_{12}^2}\,\mbox{ and }\,H = \frac{1}{2}\frac{g_{11}h_{22}-2g_{12}h_{12}+g_{22}h_{11}}{g_{11}g_{22}-g_{12}^2}\,.
\end{equation}
{Thus, the expressions for $K$ and $H$ are the same in both $\mathbb{I}^3$ and $\mathbb{I}_{\mathrm{p}}^3$, the only difference being that $\det g_{ij}$ is positive in $\mathbb{I}^3$ and negative in $\mathbb{I}_{\mathrm{p}}^3$.}

\begin{example}
Let $\mathbf{x}=(u^1,u^2,f(u^1,u^2))$ be an admissible surface parameterized in its normal form. Its first fundamental form is given by Eq. (\ref{eq::1stFundFormNormalPar}). On the other hand, we have $\mathbf{x}_i=\mathbf{e}_i+f_i\,\mathbf{e}_3$, $\mathbf{x}_{ij}=f_{ij}\,\mathbf{e}_3$, and then
\begin{enumerate}[(a)]
\item in $\mathbb{I}^3$, we find $\mathbf{x}_1\times\mathbf{x}_2=(-f_1,-f_2,1)$ and $\tilde{\mathbf{x}}_1\times\tilde{\mathbf{x}}_2=(0,0,1)$. Consequently, the second fundamental form is $\mathrm{II}=f_{11}(\mathrm{d}u^1)^2+2f_{12}\,\mathrm{d}u^1\mathrm{d}u^2+f_{22}(\mathrm{d}u^2)^2$ and the Gaussian and mean curvatures are respectively
\begin{equation}
K = f_{11}f_{22}-f_{12}^2\,\mbox{ and }\,H = \frac{f_{11}+f_{22}}{2}\,;
\end{equation}
\item in $\mathbb{I}_{\mathrm{p}}^3$, we find $\mathbf{x}_1\times_1\mathbf{x}_2=(-f_1,f_2,1)$ and $\tilde{\mathbf{x}}_1\times_1\tilde{\mathbf{x}}_2=(0,0,1)$. Consequently, the second fundamental form is $\mathrm{II}=f_{11}(\mathrm{d}u^1)^2+2f_{12}\mathrm{d}u^1\mathrm{d}u^2+f_{22}(\mathrm{d}u^2)^2$ and the Gaussian and mean curvatures are respectively
\begin{equation}
K = f_{12}^2-f_{11}f_{22}\,\mbox{ and }\,H = \frac{f_{11}-f_{22}}{2}\,.
\end{equation}
\end{enumerate}
\label{exe::KandHInNormalPar}
\end{example}
\begin{proposition}
Every admissible pseudo-isotropic minimal surface $S\subset\mathbb{I}_{\mathrm{p}}^3$, i.e., zero mean curvature surfaces, can be parameterized as
\begin{equation}
\mathbf{x}(u^1,u^2)=\left(u^1,u^2,f(u^1+u^2)+g(u^1-u^2)\right),
\end{equation}
where $f$ and $g$ are smooth real functions.
\end{proposition}
\begin{proof}
When written in its normal form $\mathbf{x}=(u^1,u^2,z(u^1,u^2))$, a minimal pseudo-isotropic surface is associated with the homogeneous wave equation $z_{11}-z_{22}=0$, whose general solution is of the form $z(u^1,u^2)=f(u^1+u^2)+g(u^1-u^2)$ for some smooth functions $f$ and $g$.
\qed
\end{proof}
\begin{remark}
The minimal surfaces in $\mathbb{I}^3$ are associated with the solution of the Laplace equation $z_{11}+z_{22}=0$. Consequently, $z$ should be the real or imaginary part of a holomorphic function, a fact that allows for a generic description of simply isotropic minimal surfaces \cite{Sachs1990}. In $\mathbb{I}_{\mathrm{p}}^3$, we just showed that every minimal surface is a special kind of an affine translation surface \cite{AydinIEJG2017,LiuPJA2013}.
\end{remark}

Spheres of parabolic type are graphs of quadratic polynomials $f=[(u^1)^2\pm(u^2)^2]/p+b_1u^1+b_2u^2+a_0$, from which it easily follows the

\begin{proposition}
Every sphere of parabolic type has constant Gaussian and mean curvatures equal to $K=1/p^2$ and $H=1/p$.\label{prop::GaussMeanCurvParabSph}
\end{proposition}

Finally, it is worth mentioning that the isotropic curvature $K$  may be also named as the \emph{relative Gaussian curvature}, in opposition to the \emph{absolute Gaussian curvature} $K_a$, which is the intrinsic curvature of the first fundamental form. In the simply isotropic geometry, $K_a$ vanishes for \emph{every} surface \cite{Sachs1990}:
\begin{equation}
K_a=\frac{1}{g_{11}}\left(\Gamma^{2}_{11,2}-\Gamma_{12,1}^{2}+\Gamma_{11}^{\ell}\,\Gamma_{\ell2}^{2}-\Gamma_{12}^{\ell}\,\Gamma_{\ell1}^{2}\right)\equiv0,
\end{equation}
where, denoting by $\mathcal{N}=(0,0,1)$ the isotropic surface normal, the coefficients  $\Gamma_{ij}^k$ are the isotropic Christoffel  symbols defined through the relation
\begin{equation}
\mathbf{x}_{ij} = \Gamma_{ij}^k\mathbf{x}_k+h_{ij}\,\mathcal{N}.
\end{equation}
We shall see in the following that, using the isotropic Gauss map, it is possible to introduce a new connection in isotropic space in a way that the intrinsic curvature is no longer trivial and is directly related to the relative Gaussian curvature. Besides the interpretation of $K$ as  the determinant of the shape operator and its relation with a new notion of connection (to be obtained from this shape operator), let us mention that $K(q)$ can be seen as the ratio between the area of a region $\xi(U)$ in $\Sigma^2$ under the Gauss map and the area of $U\subset S$ in the limit $U\to \{q\}$ \cite{Sachs1990}, p. 178, in analogy with Euclidean geometry.

\subsection{Principal curvatures and totally umbilical surfaces}

If we fix $q\in S$, then we can see $\kappa_n$ as a function on the set of unit velocity vectors in the isotropic metric, i.e.,  $\kappa_n(q,\cdot): {\Sigma^1}\subset T_qS\to\mathbb{R}$. In $\mathbb{I}^3$, the unit {circle $\Sigma^1$} in $T_qS$ is compact and then $\kappa_n$ has both a maximum $\kappa_1$ and a minimum $\kappa_2$. Notice, if $T_qS$ were isotropic, $\Sigma^1$ would be a circle of parabolic type \cite{Sachs1987}, then non-compact. The extremal values of $\kappa_n$  are the \emph{principal curvatures} and they are precisely the eigenvalues of the shape operator. Therefore, it is possible to write $K=\kappa_1\kappa_2$ and $H=\frac{1}{2}(\kappa_1+\kappa_2)$. On the other hand, in $\mathbb{I}_{\mathrm{p}}^3$ the unit {circle $\Sigma^1$} is no longer compact (in coordinates, ${\Sigma^1}=\{x^2-y^2=\pm 1\}$) and then $\kappa_n$ may fail to have both a maximum and a minimum. As a consequence, the shape operator may fail to be diagonalizable (see Sect. \ref{sec::SurfTheoryPseudoIsoSpc}), as happens in Lorentzian geometry as well \cite{LopesIEJG2014}.

It may happen that all the directions in $T_qS$ are eigenvectors of the shape operator, which occurs precisely when $\mathrm{I}$ and $\mathrm{II}$ are multiple. Then, we have

\begin{definition}
A point $q$ where the first and second fundamental forms are proportional is said to be an \emph{umbilic point}, i.e., $q$ is umbilic when $
\mathrm{II}=\lambda\,\mathrm{I}$ at $q$. A surface whose every point is umbilic is said to be \emph{totally umbilical}.
\end{definition}

In $\mathbb{I}^3$, the only totally umbilical surfaces are spheres of parabolic type and non-isotropic planes \cite{Sachs1990}, p. 171. Analogously, in $\mathbb{I}_{\mathrm{p}}^3$ the following proposition is valid.
\begin{proposition}\label{prop::SphrParAreTU}
The only totally umbilical surfaces in $\mathbb{I}_{\mathrm{p}}^3$ are spheres of parabolic type and non-isotropic planes.
\end{proposition}
\begin{proof}
Assume that the surface is given in its normal form, see Example \ref{exe::KandHInNormalPar}. In order to be totally umbilical, we must have $f_{12}=0$ and $f_{11}=-f_{22}$. From the first equation we deduce that $f(u^1,u^2)=F_1(u^1)-F_2(u^2)$ for some $F_i(u^i)$. On the other hand, $f_{11}=-f_{22}$ implies $F_1''(u^1)=F_2''(u^2)$ and, therefore, there exists a constant $2c_0$ such that $F_i''(u^i)=2c_0$. So, we have $F_i(u^i)=c_0(u^i)^2+b_iu^i+a_i$, for some constants $a_i,b_i$. In short, $S$ can be parameterized by 
\begin{equation}
\mathbf{x}(u^1,u^2)=(u^1,u^2,c_0\left[(u^1)^2-(u^2)^2\right]+b_1u^1-b_2u^2+a_1-a_2).
\end{equation}
Thus, $S$ is a sphere if $1/2p_0=c_0\not=0$ or a non-isotropic plane if $c_0=0$.
\qed
\end{proof}

\section{Surface theory in simply isotropic space}

\subsection{Relative differential geometry in simply isotropic space}
\label{subsect::RelDGsimplyIsoSpc}

{Following \cite{PavkovicRadJAZU1990},} we introduce a new connection in $\mathbb{I}^3$ whose coefficients $\Xi_{ij}^k$ come from
\begin{equation}
\mathbf{x}_{ij} = \Xi_{ij}^k\,\mathbf{x}_k+\rho_{ij}\,\xi\,.\label{def::RelCoeffXi}
\end{equation}
The coefficient $\rho_{ij}$ is unequivocally defined since $\{\mathbf{x}_1,\mathbf{x}_2,\xi\}$ is always a basis for $\mathbb{R}^3$. Indeed, using Eq. (\ref{def::DefXandXij}) and the definition of $\xi$ in Eq. (\ref{def::IsoGaussMap}), one has
\begin{equation}\label{eq::LIconditionForGaussMap}
\langle\mathbf{x}_1\times\mathbf{x}_2,\xi\rangle=\langle(X_{23},X_{31},X_{12}),\xi\rangle=\frac{(X_{23})^2+(X_{31})^2+(X_{12})^2}{2X_{12}}>0.
\end{equation}

In addition, the coefficient $\rho_{ij}$ satisfies
\begin{equation}
h_{ij} = \langle \mathbf{x}_{ij},N_h\rangle = \rho_{ij}\langle \xi,N_h\rangle = \rho_{ij}\langle \tilde{\xi}+\xi^3\mathbf{e}_3,\tilde{\xi}+\mathbf{e}_3\rangle = \rho_{ij}(\Vert\tilde{\xi}\Vert^2+\xi^3)
\end{equation}
and then
\begin{equation}\label{eq::RelationRHOijHij}
\rho_{ij} = \frac{h_{ij}}{\Vert\tilde{\xi}\Vert^2+\xi^3}\,\Rightarrow \rho_{ij}=\rho_{ji}\mbox{ and }\Xi_{ij}^k=\Xi_{ji}^k.
\end{equation}
Observe that $2(\Vert\tilde{\xi}\Vert^2+\xi^3)=[(X_{23})^2+(X_{31})^2+(X_{12})^2)](X_{12})^{-2}>0$ and, therefore, there is no singularity in the expression for $\rho_{ij}$.

\begin{proposition}
The coefficients $\Xi_{ij}^k$ relate to $\Gamma_{ij}^k$ according to 
\begin{equation}
\Xi_{ij}^k = \Gamma_{ij}^k + g^{k\ell}x_{\ell}^3\,\rho_{ij} {=\Gamma_{ij}^k + g^{k\ell}x_{\ell}^3\,\frac{h_{ij}}{\Vert\tilde{\xi}\Vert^2+\xi^3}.}
\end{equation}
\end{proposition}
\begin{proof}
Assume the notation  $g_{ab,c}:=\partial g_{ab}/\partial u^c$. From $g_{ab}=\langle\tilde{\mathbf{x}}_a,\tilde{\mathbf{x}}_b\rangle$, we have
\begin{eqnarray}
g_{ab,c} & = & \langle\tilde{\mathbf{x}}_{ac},\tilde{\mathbf{x}}_b\rangle+\langle\tilde{\mathbf{x}}_a,\tilde{\mathbf{x}}_{bc}\rangle= \left\langle\Xi_{ac}^{d}\tilde{\mathbf{x}}_{d}+\rho_{ac}\tilde{\xi}\,,\,\tilde{\mathbf{x}}_b\right\rangle+\left\langle\tilde{\mathbf{x}}_a\,,\,\Xi_{bc}^{d}\tilde{\mathbf{x}}_{d}+\rho_{bc}\tilde{\xi}\right\rangle\nonumber\\
& = & \Xi_{ac}^{d}\langle\tilde{\mathbf{x}}_{d},\tilde{\mathbf{x}}_b\rangle+\Xi_{bc}^{d}\langle\tilde{\mathbf{x}}_a,\tilde{\mathbf{x}}_{d}\rangle+\rho_{ac}\langle\tilde{\xi},\tilde{\mathbf{x}}_b\rangle+\rho_{bc}\langle\tilde{\mathbf{x}}_a,\tilde{\xi}\rangle\,.
\end{eqnarray}
Now, {using that
    $0=\langle N_h,w_q\rangle=\langle\tilde{\xi}+\mathbf{e}_3,w_q\rangle=\langle\xi,w_q\rangle_z+w_q^3$, for any $w_q\in T_qS$, one finds}
\begin{equation}
g_{ab,c} = \Xi_{ac}^{d}\,g_{db}+\Xi_{bc}^{d}\,g_{ad}-\rho_{ac}\,x^3_b-\rho_{bc}\,x^3_a\,.
\end{equation}
Finally, computing $g_{i\ell,j}+g_{\ell j,i}-g_{ij,\ell}$ and using the symmetry $\Xi_{ab}^c=\Xi_{ba}^c$, 
\begin{equation}
\Xi_{ij}^k = \frac{g^{k\ell}}{2}\left(g_{i\ell,j}+g_{\ell j,i}-g_{ij,\ell}\right)+g^{k\ell}x^3_{\ell}\,\rho_{ij} = \Gamma_{ij}^k+g^{k\ell}x^3_{\ell}\,\rho_{ij}\,.
\end{equation}
\qed
\end{proof}

\begin{remark}
{An alternative proof for the proposition above can be provided by using Eq. (8) of \cite{PavkovicRadJAZU1990}. Indeed, first notice that $\xi=-x_j^3g^{jk}\mathbf{x}_k+(\xi^3+x_i^3x_j^3g^{ij})\mathbf{e}_3$. Then, expressing $g_{ij}$ and $g^{ij}$ using $x_k^{\ell}$, we find $g^{1\ell}x_{\ell}^3=X_{12}^{-2}(x_2^1X_{31}-x_2^2X_{23})$, $g^{2\ell}x_{\ell}^3=X_{12}^{-2}(-x_1^1X_{31}+x_1^2X_{23})$, and finally $x_i^3x_j^3g^{ij}=X_{12}^{-2}[(X_{31})^2+(X_{23})^2]$. The result then follows by substituting these expressions in Eq. (8) of \cite{PavkovicRadJAZU1990}.} 
\end{remark}

\begin{definition}[{\cite{PavkovicRadJAZU1990}}]
We say that a curve $\gamma:I\to S$ is a \emph{relative geodesic} (or \emph{$r$-geodesic}) if the acceleration vector $\gamma''$ is parallel to $\xi$.
\end{definition}

The coefficients $\Xi_{ij}^k$ define a covariant derivative $\nabla^{r}$ through
\begin{equation}
\nabla^{r}_{\mathbf{x}_i}\mathbf{x}_j = \Xi_{ij}^k\,\mathbf{x}_k\label{def::RelCovDiff}
\end{equation}
and then, for any $v_q=v^i\mathbf{x}_i,w_q=w^i\mathbf{x}_i\in T_qS$, one has
\begin{equation}
\nabla^{r}_{v_q}w_q = [v_q(w^k)+v^iw^j\,\Xi_{ij}^k]\,\mathbf{x}_k\,.
\end{equation}
We may refer to $\nabla^{r}$ as the \emph{relative connection} or \emph{r-connection}. Now, computing the intrinsic acceleration $\nabla^{r}_{\gamma'}\,\gamma'$ we may deduce the standard result below. 
\begin{proposition}
A curve $\gamma:I\to S$ is an r-geodesic if and only if
\begin{equation}
\nabla^{r}_{\gamma'}\,\gamma'=0\Leftrightarrow
\frac{\mathrm{d}^2u^k}{\mathrm{d}\,t^2}+\Xi_{ij}^k\,\frac{\mathrm{d}u^i}{\mathrm{d}t}\frac{\mathrm{d}u^j}{\mathrm{d}t}=0\,,\,k\in\{1,2\},
\end{equation}
where $\gamma(t)=\mathbf{x}(u^1(t),u^2(t))$ and $\mathbf{x}$ is a local parameterization of $S$.
\end{proposition}

\begin{remark}
The relative connection $\Xi_{ij}^k$ is not metric with respect to the induced isotropic metric since it is degenerate, in which case it is possible to have more than one symmetric and metric connection \cite{VogelAM1965}. As a corollary, $r$-geodesics are not necessarily parameterized by arc-length, see example \ref{Exe::r-GeodOnASphere}.
\end{remark}

\begin{example}[r-geodesics on a plane]
Let $S$ be a non-isotropic plane. Clearly, any straight line $t\mapsto q+t\,\mathbf{u}$, $\mathbf{u}\in T_qS\cong S$, is an $r$-geodesic. By the existence and uniqueness theorem for ODE's, these are the only $r$-geodesics in $S$.
\qed
\end{example}

Geodesics according to the Levi-Civita connection are easy to find: they are the intersection with isotropic planes, since the length minimization property is defined with respect to $\mathrm{d}s^2=\mathrm{d}x^2\pm\mathrm{d}y^2$ on the top view plane {\cite{PavkovicGM1980,Sachs1990}}. On the other hand, the computation of $r$-geodesics is not so trivial.
\begin{example}[r-geodesics on a sphere of parabolic type]
\label{Exe::r-GeodOnASphere}
Let $S$ be the sphere of parabolic type $\Sigma^2(p)=\{z=\frac{p}{2}-\frac{1}{2p}(x^2+y^2)\}$ centered at the origin. In $\mathbb{E}^3$, it is known that the geodesics on a sphere can be obtained by intersecting it with planes passing through its center. Now, we show the same for $S=\Sigma^2(p)$. 

The intersection $S\cap\Pi_{a,b}=\{z=-ax-by\}$, $\Pi_{a,b}$ non-isotropic, is the curve
\begin{equation}
\gamma(t)=p\left(R\cos\theta(t)+a,R\sin\theta(t)+b,-a^2-b^2-R[a\cos\theta(t)+b\sin\theta(t)]\right)\,,
\end{equation}
where $R=\sqrt{1+a^2+b^2}$ (\footnote{Indeed, if $(x^*,y^*,z^*)\in\Pi_{a,b}\cap S$, then from $-ax^*-bx^*=p/2-[(x^*)^2+(y^*)^2]/2p$ we find $(x^*-ap)^2+(y^*-bp)^2=p^2(1+a^2+b^2)$.}). In order to have $\gamma''\parallel \gamma/p$, $\xi\circ\gamma=\gamma/p$, it is enough to find a function $\theta(t)$ such that $\gamma\times\gamma''=0$. This leads to 
\begin{equation}
\left[\theta''(R+a\cos\theta+b\sin\theta)-\theta'\,^2(a\sin\theta-b\cos\theta)\right](R\,a,R\,b,R)=0\,.\label{eqAux::CondAccGammaParallelXi}
\end{equation}
By writing the constants $a,b$ as $(a,b)=\rho(\cos\phi,\sin\phi)$, with $\rho=\sqrt{a^2+b^2}<R$, we find $R+a\cos\theta+b\sin\theta=R+\rho\cos(\theta-\phi)>0$ and then Eq. (\ref{eqAux::CondAccGammaParallelXi}) gives
\begin{equation}
\frac{\mathrm{d}^2\theta}{\mathrm{d}\,t^2}=\frac{\rho\sin(\theta-\phi)}{R+\rho\cos(\theta-\phi)}\left(\frac{\mathrm{d}\theta}{\mathrm{d}t}\right)^2.\label{eqAux::EDforTheta}
\end{equation}
Now, define $\Theta=\theta'$ and observe that $\theta''=\Theta'=\dot{\Theta}\,\theta'$, where a prime and a dot denote differentiation with respect to $t$ and $\theta$, respectively. Then, Eq. (\ref{eqAux::EDforTheta}) can be alternatively written as a first order differential equation
\begin{equation}
\frac{\mathrm{d}\Theta}{\mathrm{d}\,\theta}=\frac{\rho\sin(\theta-\phi)}{R+\rho\cos(\theta-\phi)}\,\Theta,\label{eqAux::EDforThetaAs1stODE}
\end{equation}
whose solution does exist and it is unique for any given initial condition. Thus, once we know a solution $F(\theta)$ of Eq. (\ref{eqAux::EDforThetaAs1stODE}), we can find $\theta(t)$ by solving $\theta'=F(\theta)$, for which it is also valid the existence and uniqueness theorem for ODE's. 

On the other hand, if $\Pi$ is an isotropic plane passing through the origin, then the intersection $\gamma=\Pi\cap S$ can be written as
\begin{equation}
\gamma(t) = \left(x(t),0,\frac{p}{2}-\frac{x^2(t)}{2p}\right),
\end{equation}
where we are assuming, without loss of generality, that $\Pi$ is the $xz$-plane. In order to have $\gamma''\parallel \gamma$, it is enough to find $x(t)$ with $\gamma\times\gamma''=0$, which leads to
\begin{equation}
\frac{\mathrm{d}^2x}{\mathrm{d}\,t^2}=-\frac{2x}{p^2+x^2}\left(\frac{\mathrm{d}x}{\mathrm{d}t}\right)^2\stackrel{X=x'}{\Longrightarrow}\frac{\mathrm{d}X}{\mathrm{d}\,x}=-\frac{2x}{p^2+x^2}X\,.
\end{equation}
By the same reasoning as before, it is possible to find a solution $x(t)$.

In short, the intersections of $S$ with  planes passing through its center can describe all the parameterized $r$-geodesics on a sphere of parabolic type.
\qed
\end{example}

\subsection{Gauss and Codazzi-Mainardi equations for the relative connection}
\label{subsec::RelativeGauss-Codazzi-MainardiEqs}

Let us exploit the equality $\mathbf{x}_{ab,c}=\mathbf{x}_{ac,b}$. We have
\begin{equation}
\mathbf{x}_{ab,c} = \left(\Xi_{ab,c}^{e}+\Xi_{ab}^{d}\,\Xi_{cd}^e-\rho_{ab}A_c^e\right)\mathbf{x}_e+\left(\rho_{ab,c}+\Xi_{ab}^d\,\rho_{cd}\right)\xi,
\end{equation}
where $\xi_c=-A_c^e\,\mathbf{x}_e$. From the coefficients of $\mathbf{x}_e$, we deduce the \emph{Gauss equation}
\begin{eqnarray}
\Xi_{ab,c}^e-\Xi_{ac,b}^e+\Xi_{ab}^d\Xi_{cd}^e-\Xi_{ac}^d\Xi_{bd}^e & = & (\rho_{ab}h_{cd}-\rho_{ac}h_{bd})g^{ed}\\ \label{eq::GaussEq_RhoH}
& = & (h_{ab}h_{cd}-h_{ac}h_{bd})\frac{g^{ed}}{\Vert\tilde{\xi}\Vert^2+\xi^3}\\ \label{eq::GaussEq_HH}
& = & (\Vert\tilde{\xi}\Vert^2+\xi^3)(\rho_{ab}\rho_{cd}-\rho_{ac}\rho_{bd})g^{ed}\,,\label{eq::GaussEq_RhoRho}
\end{eqnarray}
where we used Eq. (\ref{eq::CoefShapeOpe}). 
On the other hand, from the coefficient of $\xi$ in $\mathbf{x}_{ab,c}-\mathbf{x}_{ac,b}=0$, we deduce the \emph{Codazzi-Mainardi equation}
\begin{equation}
\rho_{ab,c}-\rho_{ac,b}+\Xi_{ab}^d\rho_{cd}-\Xi_{ac}^d\rho_{bd} = 0\,.\label{eq::GaussMainardi}
\end{equation}

Finally, let us introduce the \emph{r-curvature tensor} as
\begin{equation}
\mathcal{R}_{ijk}^{\ell} = \Xi_{ij,k}^{\ell}-\Xi_{ik,j}^{\ell}+\Xi_{ij}^s\Xi_{ks}^{\ell}-\Xi_{ik}^s\Xi_{js}^{\ell}\,.
\end{equation}
Now, using that $g_{fe}\,g^{ed}=\delta_f^d$ and making $(a,b,c)=(1,1,2)$ in the right-hand side of the Gauss equation Eq. (\ref{eq::GaussEq_HH}), we find 
\begin{equation}
g_{ef}\mathcal{R}_{abc}^d=\frac{g_{fe}\,g^{ed}}{\Vert\tilde{\xi}\Vert^2+\xi^3}(h_{ab}h_{cd}-h_{ac}h_{bd})=\frac{\delta_{f}^{d}}{\Vert\tilde{\xi}\Vert^2+\xi^3}(h_{ab}h_{cd}-h_{ac}h_{bd}),
\end{equation}
which implies
\begin{equation}\label{eq::IsoRdabc}
\mathcal{R}_{dabc}:=g_{ed}\mathcal{R}_{abc}^e=\frac{h_{ab}h_{cd}-h_{ac}h_{bd}}{\Vert\tilde{\xi}\Vert^2+\xi^3}\,.
\end{equation}
Then, we deduce that 
\begin{equation}
K =\frac{\det(h)}{\det(g)} = (\Vert\tilde{\xi}\Vert^2+\xi^3)\times\frac{\mathcal{R}_{2112}}{g_{11}g_{22}-(g_{12})^2},
\end{equation}
which, from $\mathbf{x}_{112}=\mathbf{x}_{121}$ and $\rho_{11}\,A_2^2-\rho_{12}\,A_1^2=(\Vert\tilde{\xi}\Vert^2+\xi^3)^{-1}g_{11}\det(L_q)$, can be alternatively rewritten as
\begin{equation}
K = \frac{\Vert\tilde{\xi}\Vert^2+\xi^3}{g_{11}}\left(\Xi^{2}_{11,2}-\Xi_{12,1}^{2}+\Xi_{11}^{\ell}\,\Xi_{\ell2}^{2}-\Xi_{12}^{\ell}\,\Xi_{\ell1}^{2}\right).
\end{equation}
Therefore, the relative Gaussian curvature only depends on the Gauss map $\xi$ and on the coefficients $\Xi_{ij}^k$ of the $r$-connection. This equation represents the Theorema Egregium for the relative Gaussian curvature according to the $r$-connection.

\section{Surface theory in pseudo-isotropic space}
\label{sec::SurfTheoryPseudoIsoSpc}

In pseudo-Euclidean geometry, depending on the properties of the induced metric, we may associate a \emph{causal character} with a surface. In $\mathbb{I}_{\mathrm{p}}^3$, any admissible surface 
is timelike (Prop. \ref{prop::PseudoIsoSurfAreTimelike}) and non-admissible surfaces are all lightlike.

Despite the fact that the shape operator $L_q$ is symmetric with respect to the induced metric,  $L_q$ may fail to be diagonalizable in the pseudo-isotropic space. This is in contrast with $\mathbb{I}^3$, where the shape operator is always diagonalizable since the induced metric is Riemannian. In $\mathbb{I}_{\mathrm{p}}^3$, the diagonalization of the shape operator depends on the existence of real
roots of its characteristic polynomial {$
C_{L_q}(\lambda) = \lambda^2-\mathrm{tr}(L_q)\,\lambda+\det(L_q)= \lambda^2-2H\,\lambda+K$.} The shape operator is diagonalizable if  $C_{L_q}$ has two distinct real roots, i.e., if $H^2-K>0$. However, if $H^2-K<0$, $L_q$ is not diagonalizable and if $H^2-K=0$, then $L_q$ may be diagonalizable or not. Finally, when $S$ is totally umbilical, $L_q$ is diagonalizable and $H^2-K=0$. Notice that we always have $H^2-K\geq0$ for surfaces in $\mathbb{I}^3$. 

\begin{remark}
Naturally, if we allow for imaginaries, all shape operators for which $H^2-K<0$ would be diagonalizable over $\mathbb{C}$. However, even with such an extension, shape operators with $H^2-K=0$ could remain not diagonalizable, as is the case of $L_q$ in example \ref{exe::SnotTU}. Here, we shall work over $\mathbb{R}$ only.
\end{remark}

\begin{example}[$H^2-K=0$, but $S$ not totally umbilical: $L_q$ not diagonalizable]\label{exe::SnotTU}
Let $S\subset\mathbb{I}_{\mathrm{p}}^3$ be the surface parameterized by {$
\mathbf{x}(u^1,u^2)=(u^1,u^1+b\,u^2,u^1u^2),\,b>0$.} 
Here $\mathbf{x}_1=(1,1,u^2)$, $\mathbf{x}_2=(0,b,u^1)$, $\mathbf{x}_1\times_1\mathbf{x}_2=(u^1-b\,u^2,u^1,b)$, and $S$ is admissible with metric
$\mathrm{I}=-b\,\mathrm{d}u^1\mathrm{d}u^2-b^2(\mathrm{d}u^2)^2\,.$ In addition, $\mathbf{x}_{11}=\mathbf{x}_{22}=\mathbf{0}$, $\mathbf{x}_{12}=(0,0,1)$, and then $\mathrm{II}=\mathrm{d}u^1\mathrm{d}u^2\,.$
The {shape operator and} the Gaussian and mean curvatures are
\begin{equation}
L_q=\left(
\begin{array}{cc}
    -\frac{1}{b} & 1 \\
    0 & -\frac{1}{b}\\
\end{array}
\right),\,K = \frac{1}{b^2}\,\mbox{ and }\,H=-\frac{1}{b}\Rightarrow H^2-K\equiv0\,.
\end{equation}
$L_q$ is not diagonalizable, since $V=[(1,0)_{\{\mathbf{x}_1,\mathbf{x}_2\}}]$, the eigenspace corresponding to the single eigenvalue $-1/b$, has dimension 1. Finally, despite that $H^2-K=0$, $S$ is not umbilical. In fact, $L_q$ is diagonalizable $\Leftrightarrow$ $S$ is umbilical. 
\qed
\end{example}
\begin{example}[$H^2-K<0$: $L_q$ non diagonalizable]
Let $S\subset\mathbb{I}_{\mathrm{p}}^3$ be the helicoidal surface {$\mathbf{x}(u^1,u^2)=(u^1\cosh(u^2),u^1\sinh(u^2),c\,u^2),\,c,u^1>0.$}
We have $\mathbf{x}_1=(\cosh(u^2),\sinh(u^2),0)$, $\mathbf{x}_2=(u^1\sinh(u^2),u^1\cosh(u^2),c)$, and $\mathbf{x}_1\times_1\mathbf{x}_2=(c\sinh(u^2),c\cosh(u^2),u^1)$. Then, $S$ is admissible with metric
$\mathrm{I}=(\mathrm{d}u^1)^2-(u^1)^2(\mathrm{d}u^2)^2\,.$ In addition, $\mathbf{x}_{11}=(0,0,0)$, $\mathbf{x}_{12}=(\sinh(u^2),\cosh(u^2),0)$, $\mathbf{x}_{22}=(u^1\cosh(u^2),u^1\sinh(u^2),0)$, and then $\mathrm{II}=-(c/u^1)\mathrm{d}u^1\mathrm{d}u^2\,.$
The Gaussian and mean curvatures are
\begin{equation}
K = \frac{c^2}{(u^1)^4}\,\mbox{ and }\,H=0\,.
\end{equation}
Finally, the shape operator is not diagonalizable, over $\mathbb{R}$, since $H^2-K<0$. We may say that the principal curvatures are complex.
\qed
\end{example}
\begin{example}[$H^2-K\geq0$: $L_q$ diagonalizable]
Let $S\subset\mathbb{I}_{\mathrm{p}}^3$ be the revolution surface {$
\mathbf{x}(u^1,u^2)=(u^1\cosh(u^2),u^1\sinh(u^2),z(u^1)),\,u^1>0.$} We have $\mathbf{x}_1=(\cosh(u^2),\sinh(u^2),z')$, $\mathbf{x}_2=(u^1\sinh(u^2),u^1\cosh(u^2),0)$, and $\mathbf{x}_1\times_1\mathbf{x}_2=(-u^1z'\cosh(u^2),-u^1z'\sinh(u^2),u^1)$. Then, $S$ is admissible with metric
$\mathrm{I}=(\mathrm{d}u^1)^2-(u^1)^2(\mathrm{d}u^2)^2\,.$ In addition, we have $\mathbf{x}_{11}=(0,0,z'')$, $\mathbf{x}_{12}=(\sinh(u^2),\cosh(u^2),0)$, and $\mathbf{x}_{22}=(u^1\cosh(u^2),u^1\sinh(u^2),0)$. Consequently, $\mathrm{II}=z''(u^1)\,(\mathrm{d}u^1)^2-u^1z'(u^1)\,(\mathrm{d}u^2)^2\,.$
The Gaussian and mean curvatures are
\begin{equation}
K = \frac{z'z''}{u^1}\,\mbox{ and }\,H=\frac{z''}{2}+\frac{z'}{2u^1}\,.
\end{equation}
Finally, $H^2-K=(z''/2-z'/2u^2)^2\geq0$. Equality occurs only for $z(u^1)=c_0(u^1)^2+c_1$, i.e., when $S$ is a sphere of parabolic type and, therefore, $L_q$ is diagonalizable. Otherwise, $H^2-K>0$ and $L_q$ is also diagonalizable.
\qed
\end{example}

\subsection{Relative differential geometry in pseudo-isotropic space}
\label{subsect::RelDGpseudoIsoSpc}

As in $\mathbb{I}^3$, the Christoffel symbols $\Gamma_{ij}^k$ {come from $\mathbf{x}_{ij} = \Gamma_{ij}^k\mathbf{x}_k +h_{ij}\,\mathcal{N}$,} where $\mathcal{N}=(0,0,1)$ is the normal to $S$ according to $\langle\cdot,\cdot\rangle_{pz}$. Now, from the equality $\mathbf{x}_{ij,k}-\mathbf{x}_{ik,j}=0$ in $\mathbb{R}^3$,  the pseudo-isotropic Gauss and Codazzi-Mainardi equations associated with the induced pseudo-isotropic metric are respectively
\begin{equation}
 \Gamma_{ij,k}^{\ell}-\Gamma_{ik,j}^{\ell}+\Gamma_{ij}^s\Gamma_{ks}^{\ell}-\Gamma_{ik}^s\Gamma_{js}^{\ell}= 0 \mbox{ and }
h_{ij,k}-h_{ik,j}+\Gamma_{ij}^{\ell}\,h_{\ell k}- \Gamma_{ik}^{\ell}\,h_{\ell j}=0. \label{eq::AbsPseudoGaussianCurvTensor}
\end{equation}
In analogy to what happens in $\mathbb{I}^3$, {the first expression in} Eq. (\ref{eq::AbsPseudoGaussianCurvTensor}) implies that the intrinsic (or absolute) Gaussian curvature associated with $\langle\cdot,\cdot\rangle_{pz}$ vanishes for any  $S\subset \mathbb{I}_{\mathrm{p}}^3$.

Now, let us introduce a new connection on $\mathbb{I}_{\mathrm{p}}^3$ whose coefficients $\Xi_{ij}^k$ {come from}
\begin{equation}
\mathbf{x}_{ij} = \Xi_{ij}^k\,\mathbf{x}_k+\rho_{ij}\,\xi^{\mathrm{p}}\,.\label{def::PseudoRelCoeffXi}
\end{equation}
The coefficient $\rho_{ij}$ is unequivocally defined whenever $\mathbf{x}_1\times_1\mathbf{x}_2$ is not lightlike in the background metric $\langle\cdot,\cdot\rangle_1$. Indeed, using Eq. (\ref{def::DefXandXij}) and $\xi^{\mathrm{p}}$ in Eq. (\ref{def::PseudoIsoGaussMap}),
\begin{equation}
\langle\mathbf{x}_1\times_1\mathbf{x}_2,\xi^{\mathrm{p}}\rangle_1=\langle(X_{23},X_{13},X_{12}),\xi^{\mathrm{p}}\rangle_1=\frac{(X_{23})^2-(X_{13})^2+(X_{12})^2}{2X_{12}}.
\end{equation}

We shall call a point $q\in S$ where $\Vert\mathbf{x}_1(q)\times_1\mathbf{x}_2(q)\Vert_1=0$ a \emph{lightlike point} of $S$. This notion should be not confused with the one coming from the induced metric since $g(\cdot,\cdot):=\langle\cdot,\cdot\rangle_{pz}\vert_{\,T_qS}$ is always timelike, as shown in Prop. \ref{prop::PseudoIsoSurfAreTimelike}.

{As in $\mathbb{I}^3$, here the coefficients $\rho_{ij}$ relate to $h_{ij}$ by
$h_{ij} = \rho_{ij}(\Vert\tilde{\xi}\Vert_1^2+\xi^3)$} and then
{$
\rho_{ij} = \frac{h_{ij}}{\Vert\tilde{\xi}\Vert_1^2+\xi^3}\,\Rightarrow \rho_{ij}=\rho_{ji}\mbox{ and }\Xi_{ij}^k=\Xi_{ji}^k$.} Observe that $2(\Vert\tilde{\xi}\Vert_1^2+\xi^3)=[(X_{23})^2-(X_{13})^2+(X_{12})^2)](X_{12})^{-2}\not=0$ outside the set of lightlike points. In this case, there is no singularity in the expression for $\rho_{ij}$. With a proof analogous to that of surfaces in $\mathbb{I}^3$, we can show that
\begin{proposition}
The coefficients $\Xi_{ij}^k$ relate to $\Gamma_{ij}^k$ according to 
\begin{equation}
\Xi_{ij}^k = \Gamma_{ij}^k + g^{k\ell}x_{\ell}^3\,\rho_{ij}{=\Gamma_{ij}^k + g^{k\ell}x_{\ell}^3\,h_{ij}(\Vert\tilde{\xi}\Vert_1^2+\xi^3)^{-1}.}
\end{equation}
\end{proposition}

The Gauss-Codazzi-Mainardi equations associated with $\Xi_{ij}^k$ in $\mathbb{I}_{\mathrm{p}}^3$ are analogous to the simply isotropic ones obtained in subsect. \ref{subsec::RelativeGauss-Codazzi-MainardiEqs}. They allow us to reinterpret the Gaussian curvature, Eq. (\ref{def::IsoGaussAndMeanCurv}), as an intrinsic curvature.

\begin{example}[r-geodesics on a sphere of parabolic type]
\label{Exe::r-GeodOnAPseudoIsoSphere}
Let $S$ be the sphere of parabolic type $\Sigma^2(p)=\{z=\frac{p}{2}-\frac{1}{2p}(x^2-y^2)\}$ centered at the origin. Its $r$-geodesics can be obtained by intersections with a plane passing through the origin. Indeed, the intersection $S\cap\Pi_{a,b}=\{z=-ax+by\}$ is the curve
\begin{equation}
\gamma(t)=p\left(R\cosh\theta+a,R\sinh\theta+b,-a^2+b^2-R[a\cosh\theta-b\sinh\theta]\right),
\end{equation}
where $\theta=\theta(t)$ and $R^2=1+a^2-b^2>0$ (\footnote{Indeed, if $(x^*,y^*,z^*)\in\Pi_{a,b}\cap S$, then from $-ax^*+bx^*=p/2-[(x^*)^2-(y^*)^2]/2p$ we find $(x^*-ap)^2-(y^*-bp)^2=p^2(1+a^2-b^2)$.}): if $R^2=1+a^2-b^2<0$, then
\begin{equation}
\gamma(t)=p\left(R\sinh\theta+a,R\cosh\theta+b,a^2-b^2+R[a\sinh\theta-b\cosh\theta]\right).
\end{equation}
When $R=0$, the intersection is a pair of lines, which are $r$-geodesics. To have $\gamma''\parallel \gamma/p$, $\xi^{\mathrm{p}}\circ\gamma=\gamma/p$, it is enough to find a function $\theta(t)$ such that $\gamma\times_1\gamma''=0$. The resulting equations can be managed in a similar fashion to those of $\mathbb{I}^3$, example \ref{Exe::r-GeodOnASphere}, by using the hyperbolic trigonometric functions instead of the usual ones. If $\Pi$ is isotropic and passes through the origin, we can also proceed as in example \ref{Exe::r-GeodOnASphere}.

In short, the intersections of $S$ with  planes passing through the origin can describe all $r$-geodesics on a pseudo-isotropic sphere of parabolic type.
\qed
\end{example}

\begin{acknowledgements}
The author would like to thank M. E. Aydin (Firat University) for useful discussions and the Departamento de Matem\'atica, Universidade Federal de Pernambuco (Recife, Brazil), where this research initiated when da Silva was a temporary lecturer.
\end{acknowledgements}

\textbf{Conflict of interest}
On behalf of all authors, the corresponding author states that there is no conflict of interest

\end{document}